\documentclass[11pt]{amsart}
\usepackage{amsfonts}
\usepackage{amssymb}
\usepackage{amsmath}
\usepackage{amsthm}
\usepackage[dvipdfmx]{graphicx,color}
\usepackage{ascmac}
\usepackage{moreverb}
\usepackage{fancybox}
\usepackage{fancyvrb}
\usepackage{comment}
\usepackage[all]{xy}

\theoremstyle{plain}
\newtheorem{theorem}{Theorem}
\newtheorem{lemma}[theorem]{Lemma}
\newtheorem{cor}[theorem]{Corollary}
\newtheorem{prop}[theorem]{Proposition}

\newtheorem{question}[theorem]{Question}

\newtheorem{fact}[theorem]{Fact}
\newtheorem{obs}[theorem]{Observation}

\theoremstyle{definition}

\theoremstyle{definition}

\newcommand{\om}{\omega}
\newcommand{\ep}{\varepsilon}

\title[Forcing with Cantor manifolds]{Effective forcing with Cantor manifolds\footnote{Work in progress.}}
\author{Takayuki Kihara}
\address[Takayuki Kihara]{Department of Mathematics, University of California, Berkeley, United States}
\email{kihara@math.berkeley.edu}
\date{}

\begin{document}
\maketitle

\begin{abstract}
A set $A$ of integers is called total if there is an algorithm which, given an enumeration of $A$, enumerates the complement of $A$, and called cototal if there is an algorithm which, given an enumeration of the complement of $A$, enumerates $A$. 
Many variants of totality and cototality have been studied in computability theory.
In this note, by an effective forcing construction with strongly infinite dimensional Cantor manifolds, which can be viewed as an effectivization of Zapletal's ``half-Cohen'' forcing (i.e., the forcing with Henderson compacta), we construct a set of integers whose enumeration degree is cototal, almost total, but neither cylinder-cototal nor telograph-cototal.
\end{abstract}

\section{Introduction}

A set $A\subseteq\om$ is total if the complement of $A$ is enumeration reducible to $A$, and cototal if  $A$ is enumeration reducible to the complement of $A$.
Recently, it is found that the notion of cototality naturally arises in symbolic dynamics, group theory, graph theory, etc.
For instance, Jeandel \cite{Jean15} showed that the set of words that appear in a minimal subshift is cototal, and Jeandel \cite{Jean15} and Thomas-Williams \cite{TW16} showed that the set of non-identity words in a finitely generated simple group is cototal.
Motivated by these examples, the notion of cototality has received increasing attention in computability theory.
For a thorough treatment of totality and cototality, see Andrews et al. \cite{cototal}.

In this note, to distinguish various notions of totality/cototality, we will utilize the notion of a {\em strongly infinite dimensional Cantor manifold},  which means a strongly infinite dimensional compactum with no weakly infinite dimensional partition.
The notions of a Cantor manifold and strong infinite dimensionality were introduced by Urysohn and Alexandrov, respectively.
These notions have been extensively studied in topological dimension theory \cite{EngBook,vMbook}.

Our construction can be naturally viewed as an effective forcing construction via the forcing $\mathbb{P}_Z$ with strongly infinite dimensional Cantor sub-manifolds of the Hilbert cube.
From this viewpoint, our forcing $\mathbb{P}_Z$ is related to the forcing with Henderson compacta, which is recently introduced by Zapletal \cite{ZapBook,Zap14} (see also \cite{PoZa12,Pol15}).
Here, a compactum is (strongly) Henderson if it is (strongly) infinite dimensional, and all of its subcompacta are either zero-dimensional or (strongly) infinite dimensional.
The notion $\mathbb{P}_Z$ is equivalent to the forcing with strongly Henderson compacta since every strongly infinite dimensional compactum contains a strongly Henderson Cantor manifold \cite{schori,Tumar}.
Zapletal \cite{Zap14} used this forcing to solve Fremlin's half-Cohen problem, asking the existence of a forcing which adds no Cohen real, but whose second iterate adds a Cohen real.
That is to say, the forcing $\mathbb{P}_Z$ is a half-Cohen forcing!

In \cite{Kih} I have proposed a problem on effectivizing Zapletal's forcing.
Thus, our work can be regarded as a path to solving my problem.

However, we shall emphasize a striking difference between Zapletal's work and ours.
The notion of a Cantor manifold does not appear in Zapletal's article \cite{Zap14}.
Indeed, as pointed out by Zapletal himself, the forcing with uncountable dimensional subcompacta of the Hilbert cube is already a half-Cohen forcing.
This means that the notion of a Cantor manifold is not necessary at all if we only aim at obtaining a half-Cohen forcing.
On the contrary, to achieve our purpose, we will make crucial use of the property being a Cantor manifold.


\section{Preliminaries}

In this note, we will use notions from computability theory \cite{OdiBook,OdiBook99} and infinite dimensional topology \cite{EngBook,vMbook}.

\subsection{Computability Theory}

As usual, the terminology ``computably enumerable'' is abbreviated as ``c.e.''
An {\em axiom} is a c.e.\ set of pairs $(n,D)$ of a natural number $n\in\om$ and (a canonical index of) a finite set $D\subseteq\om$.
A set $A\subseteq\om$ is {\em enumeration reducible to} $B\subseteq\om$ (written as $A\leq_eB$) if there is an axiom $\Psi$ such that $n\in A$ if and only if there is $D\subseteq B$ such that $(n,D)\in\Psi$.
In this case, we write $\Psi(B)=A$, that is, $\Psi$ is thought of as a function from $\mathcal{P}(\om)$ to $\mathcal{P}(\om)$, and thus, $\Psi$ is often called an {\em enumeration operator}.
The notion of enumeration reducibility induces an equivalence relation $\equiv_e$, and each $\equiv_e$-equivalence class is called an {\em enumeration degree} (or simply an {\em $e$-degree}).
The $e$-degrees form an upper semilattice $(\mathcal{D}_e,\leq,\vee)$.

\subsubsection{Cototality}

A set $A\subseteq\om$ is {\em total} if $A^\complement\leq_eA$, and {\em cototal} if $A\leq_eA^\complement$.
An $e$-degree is {\em total} ({\em cototal}) if it contains a total (cototal) set.
Note that the notion of totality can be defined via a total function on $\om$.
We consider the following notions for a function $g:\om\to\om$ and $b\in\om$:
\begin{align*}
{\rm Graph}(g)&=\{\langle n,m\rangle:g(n)=m\},\\
{\rm Cylinder}(g)&=\{\sigma\in\om^{<\om}:\sigma\prec g\},\\
{\rm Graph}_b(g)&=\{2\langle n,m\rangle:g(n)\not=m\}\cup\{2\langle n,m\rangle+1:m\geq b\mbox{ and }g(n)=m\}.
\end{align*}
Here, $\langle\cdot,\cdot\rangle$ is the standard pairing function, that is, an effective bijection between $\om^2$ and $\om$, and we write $\sigma\prec g$ if $\sigma$ is an initial segment of $g$, that is, $\sigma(n)=g(n)$ for $n<|\sigma|$.
As usual, every finite string on $\om$ is identified with a natural number via an effective bijection between $\om^{<\om}$ and $\om$.

Let $G$ be either ${\rm Graph}$, ${\rm Cylinder}$, or ${\rm Graph}_0$.
Then, for an $e$-degree is total if and only if it contains $G(f)$ for a function $f:\om\to\om$.
We now consider the following variants of cototality:
\begin{enumerate}
\item An $e$-degree is {\em graph-cototal} if it contains ${\rm Graph}(f)^\complement$ for some function $f:\om\to\om$.
\item An $e$-degree is {\em cylinder-cototal} if it contains ${\rm Cylinder}(f)^\complement$ for some function $f:\om\to\om$.
\item An $e$-degree is {\em telograph-cototal} if it contains ${\rm Graph}_b(f)$ for some function $f:\om\to\om$ and $b\in\om$.
\end{enumerate}

We have the following implications:

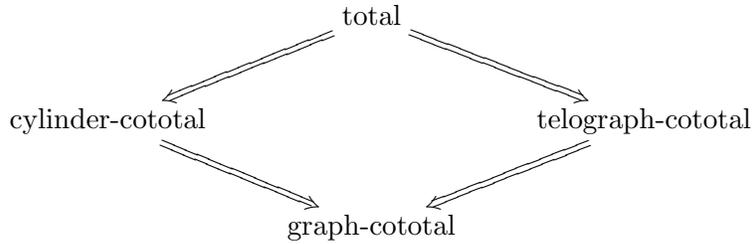
\[
\xymatrix{
& \mbox{total} \ar @{=>} [dl]  \ar @{=>} [dr]  & \\
\mbox{cylinder-cototal} \ar @{=>} [dr] & & \mbox{telograph-cototal} \ar @{=>} [dl] \\
& \mbox{graph-cototal} &
}
\]

All of the above implications are strict \cite{NTU}, and all of the above properties imply cototality \cite{cototal}.
An $e$-degree $\mathbf{a}$ is {\em almost total} if for any $e$-degree $\mathbf{b}$, either $\mathbf{b}\leq\mathbf{a}$ holds or $\mathbf{a}\vee\mathbf{b}$ is total.
Clearly, totality implies almost totality.
Although the notion of almost totality is found to be quite useful in various contexts (see \cite{Mil04,GKN,KP}), the relationship with cototality has not been studied yet. 

\subsubsection{Continuous degrees}

Fix a standard open basis $(B_e)_{e\in\om}$ of the Hilbert cube $[0,1]^\om$.
For an oracle $z\subseteq\om$, we say that a set $P\subseteq[0,1]^\om$ is $\Pi^0_1(z)$ if there is a $z$-c.e.\ set $W\subseteq\om$ such that $P=[0,1]^\om\setminus\bigcup_{e\in W}B_e$.
We identify each point $x\in[0,1]^\om$ with its coded neighborhood basis:
\[{\rm Nbase}(x)=\{e\in\om:x\in B_e\}.\]
We write $B_D=\bigcap_{d\in D}B_d$.
Note that $D\subseteq{\rm Nbase}(x)$ if and only if $x\in B_D$.
An enumeration degree $\mathbf{a}$ is called a {\em continuous degree} if $\mathbf{a}$ contains ${\rm Nbase}(x)$ for some $x\in[0,1]^\om$.
The notion of a continuous degree is introduced by Miller \cite{Mil04}, and it has turned out to be a very useful notion (see \cite{daymiller,GKN,KP}).

\begin{fact}[\cite{Mil04,cototal}]
A continuous degree is almost total, and cototal.
\end{fact}

\subsection{Infinite Dimensional Topology}

In this note, all spaces are assumed to be separable and metrizable.
A {\em compactum} is a compact metric space, and a {\em continuum} is a connected compactum.

\begin{obs}[see Section \ref{sec:appendix}]\label{obs:continuum-open}
Any nonempty open subset of a continuum is not zero-dimensional.
\end{obs}

A space $\mathcal{S}$ is an {\em absolute extensor for $\mathcal{X}$} if for any closed set $P\subseteq\mathcal{X}$, every continuous function $f:P\to\mathcal{S}$ can be extended to a continuous function $g:\mathcal{X}\to\mathcal{S}$.
A space $\mathcal{X}$ is {\em at most $n$-dimensional} if the $n$-sphere $\mathbb{S}^n$ is an absolute extensor for $\mathcal{X}$.
A space $\mathcal{X}$ is {\em countable dimensional} if it is a countable union of finite dimensional subspaces.

To introduce the notion of strong infinite dimensionality, we first recall the Eilenberg-Otto characterization of topological dimension.
Given a disjoint pair $(A,B)$ of closed subsets of $\mathcal{X}$, we say that $L$ is a {\em partition of $\mathcal{X}$ between $A$ and $B$} if there is a disjoint pair $(U,V)$ of open sets in $\mathcal{X}$ such that $A\subseteq U$, $B\subseteq V$, and $L=\mathcal{X}\setminus(U\cup V)$.
An {\em essential} sequence in $\mathcal{X}$ is a sequence of pairs of disjoint closed sets in $\mathcal{X}$ such that for any sequence $(L_i)_{i\leq n}$, if $L_i$ is a partition of $\mathcal{X}$ between $A_i$ and $B_i$, then $\bigcap_iL_i\not=\emptyset$. 
Then, a space $\mathcal{X}$ is at most $n$-dimensional if and only if $\mathcal{X}$ has no essential sequence of length $n+1$.

A space $\mathcal{X}$ is {\em strongly infinite dimensional} (abbreviated as ${\sf SID}$) if $\mathcal{X}$ has an essential sequence of length $\om$.
No ${\sf SID}$ space is countable dimensional.
Note that the Hilbert cube $[0,1]^\om$ is ${\sf SID}$.

A space $\mathcal{X}$ is {\em hereditarily strongly infinite dimensional} (abbreviated as ${\sf HSID}$) if $\mathcal{X}$ is ${\sf SID}$, and any subspace of $\mathcal{X}$ is either zero-dimensional or ${\sf SID}$.
We also say that $\mathcal{X}$ is {\em strongly Henderson} if $\mathcal{X}$ is ${\sf SID}$, and any compact subspace of $\mathcal{X}$ is either zero-dimensional or ${\sf SID}$.
Clearly, every ${\sf HSID}$ space is strongly Henderson.

A space $\mathcal{X}$ is {\em ${\sf SID}$-Cantor manifold} if $\mathcal{X}$ is compact, ${\sf SID}$, and for any disjoint pair of nonempty open sets $U,V\subseteq\mathcal{X}$, $\mathcal{X}\setminus (U\cup V)$ is ${\sf SID}$.
In other words, every partition of $\mathcal{X}$ between nonempty sets is ${\sf SID}$.


\section{Main Theorem}

For an ordinal $\alpha$, let $L_\alpha$ be the $\alpha$-th rank of G\"odel's constructible universe.
By $\om_\alpha^{\rm CK}$, we denote the $\alpha$-th ordinal which is admissible or the limit of admissible ordinals.
In this note, we will show the following:

\begin{theorem}\label{main-theorem}
There is a continuous degree $\mathbf{d}\in L_{\om_{\om}^{\rm CK}+1}\setminus L_{\om_{\om}^{\rm CK}}$ such that the following holds:
\begin{enumerate}
\item $\mathbf{d}$ is not cylinder-cototal.
\item For any $\mathbf{a}\leq\mathbf{d}$, if $\mathbf{a}$ is telograph-cototal, then $\mathbf{a}\in L_{\om_{\om}^{\rm CK}}$.
\end{enumerate}
\end{theorem}

\begin{cor}
There is a set of integers whose enumeration degree is cototal, almost total, but neither cylinder-cototal nor telograph-cototal.
\end{cor}

\subsection{Technical Tools}

To prove Theorem \ref{main-theorem}, we need to effectivize a few facts in infinite dimensional topology.
It is known that every ${\sf SID}$ compact metrizable space has an ${\sf HSID}$ compact subspace.
The construction in \cite{schori} (see also van Mill \cite[Theorem 3.3.10]{vMbook}) shows the existence of an ${\sf HSID}$ $\Pi^0_1(\mathcal{O})$ subset of $[0,1]^\om$, where $\mathcal{O}$ is a $\Pi^1_1$-complete subset of $\om$, and a modification of the construction gives a strongly Henderson $\Pi^0_1$ compactum.

\begin{prop}[see Section \ref{sec:appendix}]\label{lem:eff-HSID}
There is a strongly Henderson $\Pi^0_1$ subset of $[0,1]^\om$.
\end{prop}

We shall also extract an effective content from the {\sf SID}-version of the Cantor Manifold Theorem \cite{Tumar}, which asserts that every ${\sf SID}$ compactum contains a ${\sf SID}$-Cantor manifold.
Let $\mathcal{O}^z$ be a $\Pi^1_1$-complete subset of $\om$ relative to an oracle $z$, that is, the hyperjump of $z$.

\begin{prop}[see Section \ref{sec:appendix}]\label{cor:eff-Cantor-manifold}
Every strongly Henderson $\Pi^0_1(z)$ subset of $[0,1]^\om$ contains a $\Pi^0_1(\mathcal{O}^z)$ ${\sf SID}$-Cantor manifold.
\end{prop}

As a corollary, if a strongly Henderson compactum $H\subseteq [0,1]^\om$ has a $\Pi^0_1$-code in $L_{\om_\alpha^{\rm CK}}$, then there is an ${\sf SID}$-Cantor manifold $C\subseteq H$ which has a $\Pi^0_1$-code in $L_{\om_{\alpha}^{\rm CK}+1}$.

The reason for the need of the hyperjump in Proposition \ref{cor:eff-Cantor-manifold} is due to non-effectivity of the notion of an absolute extensor\footnote{Given a closed subset $P$ of a compactum $X$, is the set of all $f:P\to\mathbb{S}^n$ extendible over $X$ Borel in $C(P,\mathbb{S}^n)$ endowed with the compact-open topology?}.
Because the Lebesgue covering dimension and the Eilenberg-Otto characterization admit a quite effective treatment, if we could prove the Cantor Manifold Theorem only using the notions of a covering or an Eilenberg-Otto separation\footnote{Of course, the separation dimension and the extension dimension coincide, but we need an instance-wise correspondence:
Given a function $f:A\to\mathbb{S}^n$, how can we obtain a sequence of pairs of disjoint closed sets in $X$ whose inessentiality in $Y$ is equivalent to extendability of $f$ over $Y$ for any compact space $Y$ such that $A\subseteq Y\subseteq X$?},
we would obtain a far more effective version of Proposition \ref{cor:eff-Cantor-manifold}.
However, to the best of our knowledge, any proof employs the Borsuk Homotopy Extension Theorem (or its variant), which requires us to deal with extendability of functions.

\subsection{Proof of Main Theorem}

We now prove Theorem \ref{main-theorem} by applying Proposition \ref{cor:eff-Cantor-manifold}.
Hereafter, we write $\mathcal{O}^n$ for the $n$-th hyperjump.

\begin{proof}[Proof of Theorem \ref{main-theorem}]
We construct a decreasing sequence $(P_s)_{s\in\om}$ of ${\sf SID}$ $\Pi^0_1(\mathcal{O}^{\alpha(s)})$ subsets of $[0,1]^\om$ where $\alpha(s)\in\om$, and $z\in\bigcap_sP_s$ satisfies the following requirements:
\begin{align*}
\mathcal{P}_{e,k}:\;&z\not=\Psi_e(\mathcal{O}^k),\\
\mathcal{N}_{e,b}:\;&(\forall f)\;[\Psi_e(z)={\rm Graph}_b(f)\; \Longrightarrow\;(\exists n)\;f\leq_T\mathcal{O}^{n}],\\
\mathcal{M}_{d,e}:\;&(\forall g)\;[\Psi_d(z)={\rm Cylinder}(g)^\complement\mbox{ and }\Psi_e\circ\Psi_d(z)=z\;\Longrightarrow\;(\exists n)\;g\leq_T\mathcal{O}^{n}].
\end{align*}
Here, $(\Psi_e)_{e\in\om}$ is a list of enumeration operators, and $\leq$ is the Turing reducibility.
We first check that these requirements ensure the desired property.
It is clear that $\mathcal{P}$-requirements ensure that $z$ is not arithmetical.
Assume that ${\rm Graph}_b(f)\leq_ez$ for some $f$ and $b$.
By the $\mathcal{N}$-requirements, $f$ is arithmetical, and so is ${\rm Graph}_b(f)$, since ${\rm Graph}_b(f)\leq_e{\rm Graph}(f)$.
Similarly, if ${\rm Cylinder}(g)^\complement\equiv_ez$ for some $g$, then by the $\mathcal{M}$-requirements, $g$ is arithmetical, and so is ${\rm Cylinder}(g)^\complement$.
However, it is impossible since the $\mathcal{P}$-requirements ensure that $z$ is not arithmetical as mentioned above.
Thus, we have ${\rm Cylinder}(g)^\complement\not\equiv_ez$ for any $g$.

We now start the construction.
We begin with a strongly Henderson $\Pi^0_1$ set $P_0\subseteq [0,1]^\om$ in Proposition \ref{lem:eff-HSID}.
Inductively assume that we have constructed a strongly Henderson $\Pi^0_1(\mathcal{O}^{\alpha(s)})$ set $P_s\subseteq P_0$.
At the beginning of stage $s$, choose a $\Pi^0_1(\mathcal{O}^{\alpha(s)+1})$ ${\sf SID}$-Cantor manifold $Q_s\subseteq P_s$ by using Proposition \ref{cor:eff-Cantor-manifold}.
Note that $Q_s$ is also strongly Henderson since $Q_s$ is a closed subspace of $P_s$.

\subsubsection{$\mathcal{P}$-strategy.}
If $s=3\langle e,k\rangle$, we proceed the following simple diagonalization strategy.
Assume that $\Psi_e(\mathcal{O}^{k})$ determines a point $y\in Q_s$.
Since $Q_s$ has at least two points, for a sufficiently small open neighborhood $U$ of $y$, $Q_s\setminus\overline{U}$ is nonempty.
Thus, by Observation \ref{obs:continuum-open}, $Q_s\setminus U\supseteq Q_s\setminus \overline{U}$ is not zero-dimensional.
Now, $Q_s$ is strongly Henderson, and so is $P_{s+1}:=Q_s\setminus U$.
Then, the requirement $\mathcal{P}_{e,k}$ is clearly fulfilled.

\subsubsection{$\mathcal{N}$-strategy.}
If $s=3\langle e,b\rangle+1$, then for any $n\in\om$ and $c>b$, consider the following:
\begin{align*}
U_c^n&=\bigcup\{Q_s\cap B_D:(\exists m\in [b,c))\;\langle 2\langle n,m\rangle+1,D\rangle\in\Psi_e\},\\
V_c^n&=\bigcup\{Q_s\cap B_D:(\forall m\in [b,c))\;\langle 2\langle n,m\rangle,D\rangle\in\Psi_e\}.
\end{align*}

Clearly, $U_c^n$ and $V_c^n$ are c.e.~open in $Q_s$ for any $n$ and $c$.
Hereafter by $\psi^z$ we denote a function such that $\Psi_e^z={\rm Graph}_b(\psi^z)$ if it exists.
For instance, $z\in U_c^n$ ensures that $\psi^z(n)\in[b,c)$, and $z\in V_c^n$ ensures that $\psi^z(n)\not\in[b,c)$.
Without loss of generality, we can always assume that $\langle 2\langle n,m\rangle +1,D\rangle\in\Psi_e$ implies $\langle 2\langle n,k\rangle,D\rangle\in\Psi_e$ for all $k\not=m$.

\medskip

\noindent
{\bf Case A1.}
There are $c,n$ such that $U^n_c\cap V^n_c$ is nonempty.
By regularity, there is a nonempty open ball $B$ whose closure is included in $U^n_c\cap V^n_c$.
By Observation \ref{obs:continuum-open}, $Q_s\cap B$ (and hence $Q_s\cap\overline{B}$) is not zero-dimensional.
Hence, $Q_s\cap\overline{B}$ is ${\sf SID}$ since $Q_s$ is strongly Henderson.
Then, define $P_{s+1}=Q_s\cap\overline{B}$ and go to stage $s+1$.

If Case A1 is applied, we have $z\in P_{s+1}\subseteq U^n_c\cap V^n_c$.
However, $z\in U^n_c\cap V^n_c$ ensures that both $\psi^z(n)\in[b,c)$ and $\psi^z(n)\not\in[b,c)$ whenever $\psi^z(n)$ is defined.
This means that $\psi^z(n)$ is ill-defined on $z\in U^n_c\cap V^n_c$.
Hence, $\mathcal{N}_{e,b}$ is satisfied.

\medskip

\noindent
{\bf Case A2.} There are $n$ and $c$ such that both $Q_s\cap U_c^n$ and $Q_s\cap V_c^n$ are nonempty.
Note that $Q_s\cap U_c^n$ and $Q_s\cap V_c^n$ are disjoint since Case A1 is not applied.
Since $Q_s$ is a ${\sf SID}$-Cantor manifold, $P_{s+1}:=Q_s\setminus(U_c^n\cup V_c^n)$ is ${\sf SID}$.
Then go to stage $s+1$.

If Case A2 is applied at some substage $n\in\om$, then $P_{s+1}\cap (U^n_c\cup V^n_c)$ is empty, and therefore $z\not\in U^n_c\cup V^n_c$.
That is, $2\langle n,m\rangle+1\not\in\Psi^z_e$ for all $m\in[b,c)$, and $2\langle n,m\rangle\not\in\Psi^z_e$ for some $m\in[b,c)$.
The former means that $\psi^z(n)\not\in[b,c)$, but the latter has to imply that $\psi^z(n)\in[b,c)$ whenever $\psi^z(n)$ is defined.
This is impossible, and therefore $\psi^z(n)$ is undefined.
Hence, $\mathcal{N}_{e,b}$ is satisfied.

\medskip

For each $m\in[b,c)$, consider the following:
\[U_{c,m}^n=\bigcup\{Q_s\cap B_D:\langle 2\langle n,m\rangle+1,D\rangle\in\Psi_e\}.\]
Note that $z\in U_{c,m}^n$ ensures that $\psi^z(n)=m$.
Clearly, $U_c^n=\bigcup_{b\leq m<c}U_{c,m}^n$.
Now consider the following set:
\[J^n_c=\{m\in [b,c):Q_s\cap U_{c,m}^n\not=\emptyset\}\] 
The condition $m\in J^n_c$ indicates that $Q_s$ has an element $z$ such that $\psi^z(n)=m$ whenever $\psi^z$ is defined.

\medskip

\noindent
{\bf Case A3.} There are $n$ and $c$ such that $Q_s\cap V_c^n$ is empty (this indicates that $\psi^z(n)\not\in[b,c)$ is never ensured for any $z\in Q_s$), and $J^n_c$ is not a singleton.
If $J^n_c$ is empty, put $P_{s+1}=Q_s$, and go to stage $s+1$.
If $J^n_c$ has at least two elements, consider
\[F=Q_s\setminus\bigcup_{m\in[b,c)}U^n_{c,m}.\]

We claim that $F$ disconnects $Q_s$.
To see this, note that $(Q_s\cap U_{c,m}^n)_{b\leq m<c}$ is pairwise disjoint.
Otherwise, $Q_s\cap U_{c,m}^n\cap U_{c,k}^n$ is nonempty for some $b\leq m<k<c$.
Note that $U_{c,m}^n\subseteq U_{c,m+1}^n$ and by our assumption on $\Psi_e$, $U_{c,k}^n\subseteq V^n_{c,m+1}$.
This implies that $Q_s\cap U^n_{m+1}\cap V^n_{m+1}$ is nonempty.
Then, however, Case A1 must be applied.
Now, choose $m\in J_n$, and then $Q_s\setminus F$ is written as the union of disjoint open sets $Q_s\cap U^n_{c,m}$ and $Q_s\cap\bigcup_{k\not=m}U^n_{c,k}$, that is, $F$ disconnects $Q_s$.
Then, as in the previous argument, since $Q_s$ is a ${\sf SID}$-Cantor manifold, $P_{s+1}:=Q_s\cap F$ is ${\sf SID}$.
Then go to stage $s+1$.

If Case A3 is applied, since $P_{s+1}\subseteq Q_s$ does not intersect with $V^n_c$, for any $z\in P_{s+1}$, we have $\psi^z(n)\in[b,c)$ whenever $\psi^z$ is defined.
If $J_n$ is empty, then $Q_s\cap U^n_{c,m}$ is empty, in particular, $z\not\in U^n_{c,m}$, for all $m\in[b,c)$.
If $J_n$ has at least two elements, our construction ensures that $P_{s+1}\cap\bigcup_mU^n_{c,m}$ is empty, and in particular, $z\not\in U^n_{c,m}$ for all $m\in[b,c)$.
This clearly implies that $\psi^z(n)\not\in[b,c)$ in both cases.
Thus, $\psi^z$ has to be undefined, and therefore $\mathcal{N}_{e,b}$ is satisfied.

\medskip

Let $I_s$ be the set of all $n\in\om$ such that $Q_s\cap U_c^n$ is empty for any $c$.
Note that for any $n\in I_s$ and $z\in Q_s$, $\psi^z(n)\in[b,\infty)$ is never ensured. 
Then for any $n\in I_s$ and $i<b$, consider the following c.e.~open set in $Q_s$:
\[W^n_i=\bigcup\{Q_s\cap B_D:(\forall j<b)[j\not=i\;\Longrightarrow\;\langle 2\langle n,j\rangle,D\rangle\in\Psi_e\}.\]
Note that $z\in W^n_i$ ensures that $\psi^z(n)\not=j$ for any $j\in b\setminus\{i\}$.

\medskip

\noindent
{\bf Case A4.}
There are $n\in I_s$ and $i,j$ with $i\not=j$ such that $W^n_i\cap W^n_j$ is nonempty.
As in Case A1, there is a nonempty open ball $C$ such that $\overline{C}\subseteq W^n_i\cap W^n_j$.
Then, define $P_{s+1}=Q_s\cap \overline{C}$, which is ${\sf SID}$ by Observation \ref{obs:continuum-open} since $Q_s$ is strongly Henderson, and go to stage $s+1$.

If case A4 is applied, we have $z\in P_{s+1}\subseteq W^n_i\cap W^n_j$.
Note that $\psi^z(n)$ is undefined on $z\in W^n_i\cap W^n_j$, because $\psi^z(n)\not=k$ for any $k<b$, whereas $\psi^z(n)\not\in[b,\infty)$ by $n\in I_s$.
Hence, $\mathcal{N}_{e,b}$ is satisfied.

\medskip

For each $n\in I_s$, consider the following: 
\[K_n=\{i<b:Q_s\cap W^n_i\not=\emptyset\}\]

\noindent
{\bf Case A5.}
There is $n\in I_s$ (i.e., $Q_s\cap U_c^n$ is empty for any $c$) such that $K_n$ in not a singleton.
If $K_n$ is empty, put $P_{s+1}=Q_s$, and go to stage $s+1$.
If $K_n$ has at least two elements, as in Case A3, $H=Q_s\setminus\bigcup_{i<b}W^n_i$ disconnects $Q_s$; otherwise Case A4 is applied.
Since $Q_s$ is a ${\sf SID}$-Cantor manifold, $P_{s+1}:=Q_s\cap H$ is ${\sf SID}$.
Then go to stage $s+1$.

If Case A5 is applied, we have $n\in I_s$, which forces that for any $z\in P_{s+1}$, $\psi^z(n)<b$ whenever $\psi^z$ is defined.
If $K_n$ is empty, then $Q_s\cap W^n_i$ is empty, and in particular, $z\not\in W^n_i$, for all $i<b$.
If $K_n$ has at least two elements, $P_{s+1}\cap\bigcup_{i<b}W^n_i$ is empty, and in particular, $z\not\in W^n_i$ for all $i<b$.
This clearly implies $\psi^z(n)\geq b$ whenever $\psi^z$ is defined in both cases.
Thus, $\psi^z$ has to be undefined, and therefore $\mathcal{N}_{e,b}$ is satisfied.

\medskip

\noindent
{\bf Case A6.}
None of the above cases is applied.
Since Case A2 is not applied, for any $n$ and $c$, either $Q_s\cap U^n_c$ or $Q_s\cap V^n_c$ is empty.
Therefore, if $n\not\in I_s$, then $Q_s\cap V^n_c$ is empty for some $c$.
Since Case A3 is not applied, $J^n_c$ has to be a singleton for such $c$.
If $n\in I_s$, since Case A5 is not applied, $K_n$ has to be a singleton.
In other words, for any $n\in\om$, one of the following holds:
\begin{enumerate}
\item There is $c$ such that $Q_s\cap V^n_c$ is empty, and $J^n_c$ is a singleton.
\item $Q_s\cap U^n_c$ is empty for all $c$, and $K_n$ is a singleton.
\end{enumerate}

Note that given a $\Sigma^0_2(\xi)$ set $S\subseteq[0,1]^\om$, we can decide $S=\emptyset$ or not by using $\xi''$, the double jump of $\xi$.
Thus, given $n$, by using $\mathcal{O}^{\alpha(s)}{}''$, we can decide whether (1) or (2) holds, and define $\Gamma(\mathcal{O}^{\alpha(s)}{}'';n)=p$, where $p$ is a unique element in $J^n_c$ if (1) holds and $c$ is the least witness, or $K_n$ if (2) holds.
Finally, we define $P_{s+1}=Q_s$ and then go to stage $s+1$.

If Case A6 is applied, we have constructed a function $\Gamma(\mathcal{O}^{\alpha(s)}{}'')$.
We claim that for any $z\in P_{s+1}=Q_s$, $\Gamma(\mathcal{O}^{\alpha(s)}{}'')=\psi^z$ whenever $\psi^z$ is defined.
Fix $z\in P_{s+1}$ and $n\in\om$, and assume that $\psi^z(n)$ is defined.
First assume that (1) holds for $n$ and $c$ is the least witness of this fact.
Then $\Gamma(\mathcal{O}^{\alpha(s)}{}'';n)=p$ is the unique element of $J^n_c$.
Since (1) holds, $Q_s\cap V^n_c$ is empty, and as before, this forces $\psi^z(n)\in[b,c)$.
Since $J_n=\{p\}$, for all $m\in[b,c)$ with $m\not=p$, $Q_s\cap U^n_{c,m}$ is empty, that is, $z\not\in U^n_{c,m}$ and therefore $\psi^z(n)\not=m$.
Thus we must have $\psi^z(n)=p$, and therefore, $\Gamma(\mathcal{O}^{\alpha(s)}{}'';n)=p=\psi^z(n)$.
Next, assume that (2) holds.
Then $\Gamma(\mathcal{O}^{\alpha(s)}{}'';n)=p$ is the unique element of $K_n$.
Since (2) holds, $Q_s\cap U^n_c$ is empty for all $c$, which forces $\psi^z(n)<b$.
Since $K_n=\{p\}$, for all $i<b$ with $i\not=q$, $Q_s\cap W^n_i$ is empty, that is, $z\not\in W^n_i$ and therefore $\psi^z(n)\not=i$.
Thus we must have $\psi^z(n)=q$, and therefore, $\Gamma_{e,b}(\mathcal{O}^{\alpha(s)}{}'';n)=p=\psi^z(n)$.
Consequently, the requirement $\mathcal{N}_{e,b}$ is fulfilled.

\subsubsection{$\mathcal{M}$-strategy}

Now, assume that $s=3\langle d,e\rangle+2$.
Hereafter we abbreviate ${\rm Cylinder}(g)^\complement$ as $C_g$, and use $\Psi$ and $\Phi$ instead of $\Psi_d$ and $\Psi_e$, respectively.
We can assume that each basic open ball $B_e$ is of the form $\prod_{i<d(e)}B(q(e);2^{-r(e)})\times\prod_{j\geq d(e)}[0,1]_j$.
Then we define the $\varepsilon$-enlargement $B_e^\varepsilon$ as $\prod_{i<d(e)}B(q(e);2^{-r(e)}+\varepsilon)\times\prod_{j\geq d(e)}[0,1]_j$.
We can also assume that if $\langle e,D\rangle\in\Psi$, then any distinct strings $\sigma,\tau\in D$ are incomparable.

\medskip

\noindent
{\it Step 1.}
We first construct a sequence $(E_t)_{t\in\om}$ of finite sets of pairwise incomparable strings.
Each node in $E_t$ is called {\em active at $t$}.
We will also define an index $e$, a finite set $I_t$ of indices, and a positive rational $m(t)$ such that we inductively ensure the following:
\[\mbox{(IH)}\qquad g\not\in[E_t]\;\Longrightarrow\;\Psi(C_g)\in\bigcup_{d\in I_{t+1}}B_d\mbox{ or }\Psi(C_g)\in B_e^{m(t)}.\]
By removing $\bigcup_{d\in I_t}B_d\cup B_e^{m(t)}$ from $Q_s$, we will eventually ensure that if $g$ extends no active string at $t$, then $\Psi(C_g)\not\in P_{s+1}$.

We first define an index $e$.
Assume that $\Psi$ is nontrivial, that is, $\Psi(C_g)\in Q_s$ for some $g$ (otherwise put $P_{s+1}=Q_s$ and go to stage $s+1$).
Under this assumption, there is $\langle e,D\rangle\in\Psi$ such that $Q_s\cap B_e$ is nonempty, and that $Q_s\setminus B_e^\varepsilon$ has a nonempty interior for some rational $\varepsilon>0$.
Put $m(0)=0$, $I_0=\emptyset$ and declare that each string $\rho\in D$ is active at $0$, that is, $E_0=D$.
Given $t$, assume that $E_t$ has already been constructed.

Choose the lexicographically least node $\rho\in E_t$ which is active at $t$.
Consider
\[\delta=\min\{{\rm dist}(B_{e},B_{d}):d\in I_t\},\]
where ${\rm dist}$ is a formal distance function.
Note that we can effectively calculate the value $\delta$.
Inductively we assume that $m(t)<\min\{\delta,\varepsilon\}$.
We now consider the situation that there is a converging $\Psi$-computation on $[\rho]$ which outputs a point $y$ such that ${\rm dist}(y,B^{m(t)}_e)>0$.
Then, $\Psi$ eventually enumerates a sufficiently small neighborhood $B_d$ of $y$ such that the formal distance between $B_d$ and $B^{m(t)}_e$ is positive.
We first consider the case that this situation holds.

\medskip

\noindent
{\bf Case B1.}
There is $\langle d,D\rangle\in\Psi$ such that $D$ has no initial segment of $\rho$, and that the formal distance between $B_d$ and $B^{m(t)}_e$ is positive.
Then define $m(t+1)=m(t)$, $I_{t+1}=I_t\cup\{d\}$.
We also define $E_{t+1}$ as the set of strings in $D$ which extends a string in $E_t$, that is,
\[E_{t+1}=\{\sigma\in D:(\exists\tau\in E_t)\;\tau\preceq\sigma\}.\]
Note that if $g\not\in[D]$ and $\Psi(C_g)$ converges, then we have $\Psi(C_g)\in B_d$.
Thus, by induction hypothesis (IH) at $t$, for any $g\not\in[E_{t+1}]$, either $\Psi(C_g)\in\bigcup_{d\in I_{t+1}}B_d$ or $\Psi(C_g)\in B_e^{m(t+1)}$ holds.

\medskip

\noindent
{\bf Case B2.}
Otherwise, for any $\langle d,D\rangle\in\Psi$, if $D$ has no initial segment of $\rho$, then ${\rm dist}(B_d,B^{m(t)}_e)=0$ (In this case, every converging $\Psi$-computation on $[\rho]$ outputs a point $y$ such that ${\rm dist}(y,B^{m(t)}_e)=0$ as discussed above).
Then, choose a rational $q$ such that $m(t)<q<\min\{\delta,\varepsilon\}$.
In this case, define $m(t+1)=q$, and $I_{t+1}=I_t$, and declare that $\rho$ is not active at $t+1$, that is, $E_{t+1}=E_t\setminus\{\rho\}$.
Note that the closure of $B_e^{m(t)}$ is included in $B_e^{m(t+1)}$, and thus, every converging $\Psi$-computation on $[\rho]$ is contained in $B_e^{m(t+1)}$.
Thus, by induction hypothesis (IH) at $t$, for any $g\not\in[E_{t+1}]$, either $\Psi(C_g)\in\bigcup_{d\in I_{t+1}}B_d$ or $\Psi(C_g)\in B_e^{m(t+1)}$ holds.

\medskip

Finally, put $m=\sup_tm(t)$, $I=\bigcup_tI_t$, and $U=\bigcup_{d\in I}B_d$.
Note that the downward closure $E^\ast$ of $E=\bigcup_tE_t$ forms a finite-branching tree, and $\emptyset'$-c.e., since one can decide which case holds by using $\emptyset'$.
Therefore the set $[E^\ast]$ of all infinite paths through $E^\ast$ is a compact $\Pi^0_1(\emptyset'')$ subset of $\om^\om$.
We also note that $[E^\ast]=\bigcap_t[E_t]$.
Our construction ensures that, by induction,
\[g\not\in[E^\ast]\;\Longrightarrow\;\Psi(C_g)\in U\mbox{ or }\Psi(C_g)\in B_e^m.\]
We assume that $P_s\cap U$ is nonempty; otherwise, put $P_{s+1}=P_s\setminus B^\varepsilon_e$.
Our construction ensures that $B_e^m\cap U$ is empty because we always choose $m(t)<\delta$.
Since $P_s$ is a ${\sf SID}$-Cantor manifold, $P_s\setminus(B_e^m\cup U)$ is ${\sf SID}$.
Then there is a ${\sf SID}$-Cantor manifold $Q_s\subseteq P_s\setminus(B_e^m\cup U)$.

\medskip

\noindent
{\it Step 2.}
Consider the following:
\[U_\sigma=\bigcup\{B_d:\langle\sigma,D\rangle\in\Phi\}.\]
We write $\varphi(x)=g$ if $\Phi^x=C_g$.
Then, $x\in U_\sigma$ means that $\varphi(x)$ does not extend $\sigma$ whenever it is defined.
Thus, if there are incomparable strings $\sigma,\tau$ such that $x\not\in U_\sigma\cup U_\tau$, then $\varphi(x)$ is undefined.
At each substage $t$, check whether $Q_s\cap \bigcap_{\sigma\in E_t}U_\sigma$ is nonempty.
Note that if it is true, there is $x\in Q_s$ such that $\varphi(x)\not\in[E_t]$ or else $\varphi(x)$ is undefined.

\medskip

\noindent
{\bf Case C1.}
If $Q_s\cap \bigcap_{\sigma\in E_t}U_\sigma$ is nonempty at some substage $t$.
Then define $P_{s+1}=Q_s\cap \bigcap_{\sigma\in E_t}U_\sigma$ and go to stage $s+1$.
Note that $P_{s+1}$ is a nonempty open subset of a ${\sf SID}$-Cantor manifold $Q_s$, and therefore, $P_{s+1}$ is ${\sf SID}$ by Observation \ref{obs:continuum-open}.
In this case, for any $x\in P_{s+1}$, whenever $\varphi(x)$ converges, $\varphi(x)\not\in[E_t]\supseteq[E^\ast]$, and therefore $\Psi(C_{\varphi(x)})\not\in Q_s\supseteq P_{s+1}$.
This forces that $\Psi\circ\Phi(x)\not=x$ for any $x\in P_{s+1}$, which ensures the requirement $\mathcal{M}_{d,e}$.

\medskip

\noindent
{\bf Case C2.}
Case C1 is not applied, and at least two of $(V_\sigma)_{\sigma\in E_t}$ is nonempty at some substage $t$, where $V_\sigma=Q_s\cap \bigcap_{\tau\in E_t\setminus\{\sigma\}}U_\tau$.
Note that $x\in V_\sigma$ means that if $\varphi(x)\in[E_t]$, then $\varphi(x)$ must extend $\sigma$.
In this case, choose $\sigma$ such that $V_\sigma$ is nonempty, and then $V_{\not=\sigma}:=\bigcup_{\tau\in E_t\setminus\{\sigma\}}V_\tau$ is also nonempty (where $x\in V_{\not=\sigma}$ means that $\varphi(x)\in[E_t]$, then $\varphi(x)$ must extend some $\tau\in E_t\setminus\{\sigma\}$).
Note that $V_{\not=\sigma}\subseteq U_\sigma$, and thus $V_\sigma\cap V_{\not=\sigma}\subseteq Q_s\cap\bigcap_{\sigma\in E_t}U_\sigma$ is empty since Case C1 is not applied.
Therefore, $Q_s\setminus (V_\sigma\cup V_{\not=\sigma})$ is ${\sf SID}$ since $Q_s$ is a ${\sf SID}$-Cantor manifold.
Then define $P_{s+1}=Q_s\setminus (V_\sigma\cup V_{\not=\sigma})$ and go to stage $s+1$.
Then for any $x\in P_{s+1}$, since $x\not\in V_\sigma\cup V_{\not=\sigma}$, $\varphi(x)\not\in[E_t]\supseteq [E^\ast]$ as mentioned above.
This forces that $\Psi\circ\Phi(x)\not=x$ for any $x\in P_{s+1}$, which ensures the requirement $\mathcal{M}_{d,e}$.

\medskip

\noindent
{\bf Case C3.}
$V_\sigma$ is empty for any $\sigma\in E_t$,  at some substage $t$.
This automatically implies that $\varphi(x)\not\in[E_t]$ for any $x\in Q_s$.
Put $P_{s+1}=Q_s$, go to stage $s+1$, and then $\Psi\circ\Phi(x)\not=x$ for any $x\in P_{s+1}$ as in the above argument.
Then, the requirement $\mathcal{M}_{d,e}$ is fulfilled.

\medskip

\noindent
{\bf Case C4.}
Otherwise, for each $t$, there is a unique $\sigma(t)\in E_t$ such that $Q_s\subseteq V_{\sigma(t)}$.
Note that $\bigcup_t\sigma(t)$ is infinite, or otherwise $\sigma(t)$ is not active for almost all $t$.
Define $\Gamma(x)=\bigcup_t\sigma(t)$.
We claim that for any $x\in Q_s$, if $\Psi\circ\Phi(x)=x$, then $\varphi(x)=\Gamma(x)$.
First note that if $\Psi\circ\Phi(x)=x$ (in this case, $\Psi(C_{\varphi(x)})=x\in Q_s$) we must have $\varphi(x)\in[E^\ast]$ as mentioned in Case C1.
Thus, for each $t$, there is a unique $\sigma\in E_t$ such that $\varphi(x)$ extends $\sigma$, that is, $x\in V_\sigma$.
By uniqueness of $\sigma$, we must have $\sigma=\sigma(t)$.
Hence, $\varphi(x)$ extends $\sigma(t)$ for any $t\in\om$.
If $\bigcup_t\sigma(t)$ is infinite, this concludes that $\varphi(x)=\Gamma(x)$.
Otherwise, $\sigma=\sigma(t)$ for almost all $t$.
This implies that $\sigma$ is not active at some stage, and then $[E^\ast]\cap[\sigma]=\emptyset$.
This contradicts our observation that $\sigma\prec\varphi(x)\in[E^\ast]$.
Thus, our claim is verified.
Put $P_{s+1}=Q_s$, and go to stage $s+1$.
Then, by the above claim, the requirement $\mathcal{M}_{d,e}$ is fulfilled.
\end{proof}


\section{Technical appendix}\label{sec:appendix}

\subsection{Proof of Observation \ref{obs:continuum-open}}

Suppose for the sake of contradiction that there is a nonempty open subset $U$ of a continuum $\mathcal{X}$ which is zero-dimensional.
Fix a point $x\in U$,
By regularity, there is an open neighborhood $V$ of $x$ such that $\overline{V}\subseteq U$, where $\overline{V}$ is the closure of $V$ in $\mathcal{X}$.
By zero-dimensionality of $U$, there is an open sets $B\subseteq\mathcal{X}$ such that $x\in B\subseteq V$ such that $U\cap \partial B=\emptyset$.
However, we have $\partial B\subseteq\overline{B}\subseteq\overline{V}\subseteq U$, and therefore, $U\cap\partial B=\partial B=\emptyset$. 
Hence, the empty set separates $\mathcal{X}$, which contradicts our assumption that $\mathcal{X}$ is a continuum.

\subsection{Proof of Proposition \ref{lem:eff-HSID}}

Let $d$ be a computable metric on the Hilbert cube.
We say that $A\subseteq[0,1]^\om$ is {\em $d$-computable} if the function $d_A:x\mapsto\inf_{y\in A}d(x,y)$ is computable.
In other words, the closure of $A$ is computable as a closed set. 
A partition $L$ of $Y$ is {\em $d$-computable} if there are disjoint open sets $U$ and $V$ in $Y$ such that $L=Y\setminus(U\cup V)$ and $U$ and $V$ are  $d$-computable in $[0,1]^\om$.

\begin{lemma}\label{lem:pi-partition}
Let $Y$ be a subspace of $[0,1]^\om$.
Assume that $(A,B)$ and $(A^\ast,B^\ast)$ are pairs of disjoint closed subsets of $[0,1]^\om$, and that $A$ ($B$, resp.)\ is included in the interior of $A^\ast$ ($B^\ast$, resp.)
If $S$ is a $d$-computable partition of $Y$ between $Y\cap A^\ast$ and $Y\cap B^\ast$, then one can effectively find a $\Pi^0_1$ partition $T\subseteq S$ of $[0,1]^\om$ between $A$ and $B$.
\end{lemma}

\begin{proof}
Assume that $S$ is of the form $Y\setminus (U^\ast\cup V^\ast)$, so that $Y\cap A^\ast\subseteq U^\ast$ and $Y\cap B^\ast\subseteq V^\ast$.
Then $B^\ast\cap U^\ast$ is empty, and thus $B\cap\overline{U^\ast}$ is empty.
Similarly, $A\cap\overline{V^\ast}$ is empty.
Therefore, $A\cup U^\ast$ and $B\cup V^\ast$ are separated.
Consider
\begin{align*}
U&=\{x\in[0,1]^\om:d(x,A\cup U^\ast)<d(x,B\cup V^\ast)\},\\
V&=\{x\in[0,1]^\om:d(x,B\cup V^\ast)<d(x,A\cup U^\ast)\},
\end{align*}
Clearly, $U$ and $V$ are disjoint.
Note that $A\cup U^\ast$ and $B\cup V^\ast$ are $d$-computable since $d(x,E\cup F)=\min\{d(x,E),d(x,F)\}$.
Thus, $U$ and $V$ are $\Sigma^0_1$.
We also note that $A\cup U^\ast\subseteq U$ and $B\cup V^\ast\subseteq V$.
Define $T=[0,1]^\om\setminus (U\cup V)$.
\end{proof}

We assume that a basic open set of the Hilbert cube is of the form $\prod_{i<n}(p_i,q_i)\times\prod_{i\geq n}I_i$ such that $p_i$ and $q_i$ are rationals, and $I_i=[0,1]$.
Fix an effective list $\mathcal{U}=(U_n)_{n\in\om}$ of open sets written as finite unions of basic open sets in the Hilbert cube, so that one can decide whether $U_m\cap U_n=\emptyset$.
For each $i\in\om$ define $A_i=\pi_i^{-1}\{0\}$ and $B_i=\pi_i^{-1}\{1\}$.
Note that one can give an effective enumeration $(U^i_n,V^i_n)_{n\in\om}$ of all disjoint pairs in $\mathcal{U}^2$ such that $A_i\subseteq U^i_n$ and $B_i\subseteq V^i_n$.

\begin{lemma}
Let $C\subseteq[0,1]$ be a computable Cantor set, let $\Lambda$ be a computable subset of $\om$, and let $j\not\in\Lambda$.
Then, there exists a uniform sequence $(S_i)_{i\in\Lambda}$ of $\Pi^0_1$ sets in $[0,1]^\om$ such that $S_i$ is a partition of $[0,1]^\om$ between $A_i$ and $B_i$, and for any compact set $M\subseteq\bigcap_iS_i$, if $C\subseteq\pi_j(M)$ then $M$ is {\sf SID}.
\end{lemma}

\begin{proof}
Assume $\Lambda=\{a(i)\}_{i\in\om}$.
Let $C\subseteq[0,1]$ be a computable homeomorphic copy of $2^\om$, and fix a computable homeomorphism $\iota:C\to 2^\om$.
Let $C_n$ be the set of all $x$ such that $\iota(x)$ contains at least $n+1$ many ones.
If $\iota(x)$ is of the form $0^{k(0)}10^{k(1)}10^{k(2)}1\dots$, we define $\hat{x}=\langle k(0),k(1),\dots\rangle$.
Note that $x\in C_n$ if and only if $\hat{x}(n)$ is defined.
For any $x\in[0,1]$, define $L_n(x)$ as follows.
\[L_n(x)=
\begin{cases}
[0,1]^\om\setminus\left(U^{a(n)}_{\hat{x}(n)}\cup V^{a(n)}_{\hat{x}(n)}\right)&\mbox{ if }x\in C_n,\\
[0,1]^\om&\mbox{ if }x\not\in C_n.
\end{cases}
\]
Note that $L_n(x)$ is $\Pi^0_1$ uniformly relative to $x$.
Therefore,
\[H_n=\{x=(x_i)_{i\in\om}:x\in L_n(x_j)\}.\]
is $\Pi^0_1$.
Let $C_{n,m}$ be the set of all $x\in C_n$ such that $\hat{x}(n)=m$.
For $E=\pi_{a(n)}^{-1}[C_n]$, note that $H_n\cap E$ is a partition between $A_{a(n)}\cap E$ and $B_{a(n)}\cap E$.
To see this, consider
\begin{align*}
U=\bigcup_{m\in\om}U^{a(n)}_{m}\cap \pi^{-1}_{a(n)}[C_{n,m}],\\
V=\bigcup_{m\in\om}V^{a(n)}_{m}\cap \pi^{-1}_{a(n)}[C_{n,m}],
\end{align*}
It is clear that $H_n\cap E=E\setminus (U\cup V)$, $A_{a(n)}\cap E\subseteq U$, $B_{a(n)}\cap E\subseteq V$, and $U$ and $V$ are disjoint.

We claim that $U$ and $V$ are $d$-computable.
Given a positive real $\ep<2^{-j}$, put $E_\ep(x_j)=B_\ep(x_j)\cap C_n$.
If $E_\ep(x_j)$ is empty, then $d(x,U)\geq\ep$.
Assume that $E_\ep(x_j)$ is nonempty.
There are two cases.
If $[x_j-\ep,x_j+\ep)$ contains a point in $C\setminus C_n$, then $E_\ep(x_j)$ intersects with $C_{n,m}$ for almost all $m$.
Choose $S$ of the form $U^{a(n)}_k$ such that $x\in S$ and $\pi^{-1}_{a(n)}[C_{n}]\subseteq S$.
Then, clearly $B_\ep(x)$ intersects with $S$.
In this case, $d(x,U)\leq \ep$.
Otherwise, consider the set $D$ of all $m$ such that $E_\ep(x_j)$ intersects with $C_{n,m}$.
Note that $D$ is finite.
Then, consider $U_D=\bigcup_{m\in D}U^{a(n)}_m\cap\pi^{-1}_{a(n)}[C_{n,m}]$.
Clearly $U_D$ is $d$-computable uniformly in $D$.
If we see $d(x,U_D)>\ep$, then $d(x,U)\geq\ep$.
If we see $d(x,U_D)<\ep$, then $d(x,U)<\ep$.
This gives a procedure computing $d(x,U)$.
Hence, $U$ is $d$-computable.
Similarly, $V$ is $d$-computable.
Thus, $H_n\cap E$ is a $d$-computable partition of $E$ between $A_{a(n)}\cap E$ and $B_{a(n)}\cap E$.

By Lemma \ref{lem:pi-partition}, there is a uniform sequence of $(S_n)_{n\in\om}$ of $\Pi^0_1$ partitions of $[0,1]^\om$ between $A_{a(n)}$ and $B_{a(n)}$ such that $S_n\subseteq H_n$.
Then, define
\[H=\bigcap_{n\in\om}S_n.\]
Clearly, $H$ is $\Pi^0_1$.
Assume that $M$ is a compact subset of $H$.
We claim that if $C\subseteq \pi_j[M]$, then $M$ is ${\sf SID}$.
If not, there exists partitions $L_n$ of the Hilbert cube between $A_{a(n)}$ and $B_{a(n)}$ such that $M\cap\bigcap_{n\in\om}L_n$ is empty.
By compactness, we can assume that $L_n$ is of the form $[0,1]^\om\setminus(U^{a(n)}_{k(n)}\cup V^{a(n)}_{k(n)})$.
Then, for $z=0^{k(0)}10^{k(1)}1\dots$, $L_n(z)=L_n$.
Since $C\subseteq \pi_j[M]$, there is $x\in M$ such that $x_j=z$.
Since $M\subseteq H\subseteq H_n$, we have $x\in L_n(z)=L_n$.
Therefore, we get $M\cap\bigcap_nL_n\not=\emptyset$, a contradiction.
\end{proof}

Let $(C_i)_{i\in\om}$ be a computable collection of pairwise disjoint Cantor set in $[0,1]$ such that each nondegenerate subinterval of $[0,1]$ contains one of the $C_i$.
Then proceed the usual proof (see van Mill \cite[Theorem 3.13.10]{vMbook}).

\subsection{Proof of Proposition \ref{cor:eff-Cantor-manifold}}

Given a space $\mathcal{X}$, we say that $L$ is a partition of $\mathcal{X}$ if there is a pair of disjoint nonempty open sets $U,V\subseteq\mathcal{X}$ such that $L=\mathcal{X}\setminus (U\cup V)$.
Recall that an {\sf SID}-Cantor manifold $\mathcal{X}$ is a ${\sf SID}$ compactum such that, if $L$ is a partition of $\mathcal{X}$, then $L$ is ${\sf SID}$.
In the construction of a Cantor manifold, we will use the notion of an absolute extensor.
However, it is known that there is no CW-complex $S$ such that a compactum $\mathcal{X}$ is {\sf SID} if and only if $S$ is an absolute extensor for $\mathcal{X}$.
Thus, instead of directly constructing an {\sf SID}-Cantor manifold, we will show the following lemma, and combine it with the existence of a strongly Henderson $\Pi^0_1$ compactum.

\begin{lemma}\label{lem:eff-Cantor-manifold}
For any $\Pi^0_1$ set $P\subseteq [0,1]^\om$, if $P$ is at least $n$-dimensional, then $P$ contains a $\Pi^0_1(\mathcal{O})$ subset which is not partitioned by an at most $(n-2)$-dimensional closed subset.
\end{lemma}

\begin{proof}
Given a sequence  of closed subsets of $[0,1]^\om$, note that $(A_i,B_i)_{i<n}$ is inessential in $P$ if and only if
\[(\exists (U_i,V_i)_{i<n})\;A_i\subseteq U_i,\;B_i\subseteq V_i\mbox{, and }\bigcap_{i<n}(P\setminus(U_i\cup V_i))=\emptyset,\]
where $(U_i,V_i)_{i<n}$ ranges over all $n$ disjoint pairs of open subsets of $[0,1]^\om$.
By complete normality of $[0,1]^\om$, we can assume that $U_i$ and $V_i$ are disjoint in $[0,1]^\om$ (see \cite[Corollary 3.1.5]{vMbook}).
Moreover, by compactness, we can assume that $U_i$ and $V_i$ are finite unions of basic open balls in $[0,1]^\om$.
The above observation shows that inessentiality is a $\Sigma^0_1$ property.
We use ${\rm Ess}((A_i,B_i)_{i<n},P)$ to denote a $\Pi^0_1$ formula saying that $(A_i,B_i)_{i<n}$ is essential in $P$.

Since $\dim(P)\geq n$, there is a sequence $(A_i,B_i)_{i<n}$ of $n$ disjoint pairs of closed subsets of $[0,1]^\om$ which is essential in $P$.
By normality of $[0,1]^\om$, for any $i<n$, there are open sets $C_i,D_i$ such that $\overline{C_i}\cap\overline{D_i}=\emptyset$, $A_i\subseteq C_i$ and $B_i\subseteq D_i$.
By compactness of $A_i$ and $B_i$, we can assume that $C_i$ and $D_i$ are finite unions of basic open balls.
Then,
\[\mathfrak{Q}=\{(A_i,B_i)_{i<n}:{\rm Ess}((A_i\cap\overline{C_i},B_i\cap\overline{D_i})_{i<n},P)\}\]
is still a nonempty $\Pi^0_1$ collection of $n$ pairs of closed subsets of $[0,1]^\om$.
By the low basis theorem, there is a low degree $\mathbf{d}$ such that $\mathfrak{Q}$ has a $\mathbf{d}$-computable element.
In other words, there is a sequence $(A_i,B_i)_{i<n}$ of $n$ disjoint pairs of $\Pi^0_1(\mathbf{d})$ subsets of $[0,1]^\om$ which is essential in $P$.

Now, we follow the usual construction of a Cantor manifold.
For $A=\bigcup_{i<n}A_i\cup B_i$, there is a continuous function $f:A\to\mathbb{S}^{n-1}$ such that $f$ does not admit a continuous extension over $P$ (see \cite[Theorem 3.6.5]{vMbook}).
Given $s$, by using $\mathcal{O}$, we can check whether $f$ admits a continuous extension over $A\cup (P_s\setminus B_e)$.
Then, proceed the usual construction (see \cite[Theorem 3.7.8]{vMbook}).
\end{proof}

\subsection{Remarks}

It is natural to ask whether one can replace $\mathcal{O}$ in Lemma \ref{lem:eff-Cantor-manifold} with an arithmetical oracle or an arbitrary PA degree.
Here, a set of integers has a PA degree if it computes a complete consistent extension of Peano Arithmetic (or ZFC).
At this time, we do not have an answer.
We here only see that $\mathcal{O}$ cannot be replaced with $\emptyset$.
We show that we have no way of obtaining even a connected $\Pi^0_1$ subset of a given $\Pi^0_1$ set of positive dimension.

\begin{obs}\label{prop:non-trivial-subcontinua}
The following are equivalent for $x\in 2^\om$
\begin{enumerate}
\item $x$ has a PA degree.
\item Every non-zero-dimensional $\Pi^0_1$ subset of $[0,1]^\om$ has a nondegenerate $\Pi^0_1(x)$ subcontinuum.
\item Every non-zero-dimensional $\Pi^0_1$ subset of $[0,1]^2$ has a nondegenerate $\Pi^0_1(x)$ subcontinuum.
\end{enumerate}
\end{obs}

\begin{proof}
To see (1)$\Rightarrow$(2), let $P$ be a $\Pi^0_1$ subset of $[0,1]^\om$ of positive dimension.
Then, let $S$ be the set of all pairs $(a,b)\in P^2$ not separated by a clopen set in $P$, that is, there is no pair of open sets $U,V$ such that $P\subseteq U\cup V$, $U\cap V\cap P=\emptyset$, and $a\in U$ and $b\in V$.
By compactness of $P$, we can assume that $U$ and $V$ range over finite unions of basic open balls in the Hilbert cube.
Moreover, by normality, we can assume that $\overline{U}\cap\overline{V}\cap P=\emptyset$.
Therefore, it is easy to check that $S$ is a nonempty $\Pi^0_1$ set.
Moreover, we know that there is $(c,d)\in S$ such that $c\not=d$.
Choose disjoint computable closed balls $B_c$ and $B_d$ such that $c\in B_c$ and $d\in B_d$.
Thus, given a PA degree $x$, $S\cap(B_c\times B_d)$ has an $x$-computable element $(a,b)$.
Now, we claim that the quasi-component $C_a$ containing $a$ in $P$ is $\Pi^0_1(x)$.
Let $\mathcal{W}$ be an effective enumeration of all $U$ such that $P\subseteq U\cup V$ and $\overline{U}\cap\overline{V}\cap P=\emptyset$.
As in the above argument, if $S$ is a clopen set in $P$, then there is $U\in\mathcal{W}$ such that $S\cap P=U\cap P$.
This shows that $C_a=P\cap\bigcap\{U\in\mathcal{W}:a\in U\}$, which is $\Pi^0_1(x)$.
Since $P$ is a compactum, $C_a$ is indeed a connected component.
Note that $C_a$ is nondegenerate since $a,b\in C_a$ and $a\not=b$.

For (3)$\Rightarrow$(1), consider $[0,1]\times P$, where $P$ is a $\Pi^0_1$ subset of a Cantor ternary set such that all of elements in $P$ have PA degrees.
Clearly, $[0,1]\times P$ is one-dimensional, but any connected subset $Q$ is included in $[0,1]\times\{z\}$ for some $z\in P$.
Assume that $Q$ is $\Pi^0_1(x)$.
Note that the projection $\pi[Q]=\{z:(\exists y\in[0,1])\;(y,z)\in Q\}$ of $Q$ is still $\Pi^0_1(x)$.
Thus, $\pi[Q]$ is a $\Pi^0_1(x)$ singleton.
Then, $x$ computes the unique element of $\pi[Q]$, which has a PA degree.
\end{proof}

The proof of (1)$\Rightarrow$(2) in Observation \ref{prop:non-trivial-subcontinua} is not uniform.
For instance, one can easily construct a computable sequence $(P_e)_{e\in\om}$ of $\Pi^0_1$ sets with two connected components $C^0_e$ and $C^1_e$ such that the $e$-th Turing machine halts, iff $C^0_e$ is a singleton, iff $C^1_e$ is not a singleton.
This implies that, to solve the problem of finding a uniform procedure that, given a $\Pi^0_1$ set of positive dimension, returns its nontrivial subcontinuum, we need at least the power of the halting problem $\mathbf{0}'$.

\subsection{Open Questions}

We here list open problems.

\begin{question}
Does there exist an almost total degree which is not graph-cototal?
\end{question}


\begin{question}
Does there exist a $\Pi^0_1(z)$ {\sf HSID} subset of $[0,1]^\om$ for a hyperarithmeical $z$?
\end{question}

\begin{question}
Can $\mathcal{O}$ in Lemma \ref{lem:eff-Cantor-manifold} be replaced with an arithmetical oracle or an arbitrary PA degree?
\end{question}

Let $\mathcal{A}_{\dim\geq n}$ be the hyperspace of at least $n$-dimensional closed subsets of the Hilbert cube equipped with the usual negative representation.
By the {\em $n$-dimensional Cantor manifold theorem}, we mean a multi-valued function ${\sf CMT}_n:\mathcal{A}_{\dim\geq n}\rightrightarrows\mathcal{A}_{\dim\geq n}$ such that for any $X$, ${\sf CMT}_n(X)$ is the set of all closed subsets $Y\subseteq X$ such that $Y$ is not partitioned by an at most $(n-2)$-dimensional closed subset.

\begin{question}
Determine the Weihrauch degree of the $n$-dimensional Cantor manifold theorem.
\end{question}

\bibliographystyle{plain}
\bibliography{continuous.bib}

\end{document}